\documentclass[psamsfont,a4paper,12pt]{amsart}

\usepackage[english]{babel}   

\usepackage{setspace}     
\usepackage{enumerate} 	
\usepackage{afterpage} 	
\usepackage[usenames,dvipsnames]{color}
\usepackage{graphicx,hyperref}	

\usepackage{amsmath} 	
\usepackage{amssymb}	
\usepackage{mathrsfs} 	
\usepackage{amsfonts} 	
\usepackage{latexsym} 	
\usepackage[latin1]{inputenc} 

\usepackage{amsthm} 	
\usepackage{stackrel}       
\usepackage{amscd} 	

\usepackage{tabularx}	
\usepackage{longtable}	

\usepackage{verbatim} 
\usepackage{blindtext} 
\hypersetup{colorlinks=true,linkcolor=blue,citecolor=magenta}
\usepackage{lipsum}

\allowdisplaybreaks

\newcommand{\C}{\mathbb{C}}

\newcommand{\R}{\mathbb{R}}

\newcommand{\Z}{\mathbb{Z}} 
\newcommand{\Q}{\mathbb{Q}}

\newcommand{\HH}{\mathbb{H}}

\def\SL{\mathop{\rm SL}}

\renewcommand{\to}{\longrightarrow}

\def\bz{{\bf z}}

\newtheorem{Thm}{Theorem}[section]
\newtheorem*{theorem*}{Theorem}
\newtheorem*{thm*}{Theorem}
\newtheorem{Lemma}[Thm]{Lemma}

\theoremstyle{definition}

\newtheorem*{defn}{Definition}  

\theoremstyle{remark}
 
\newtheorem*{rmk}{Remark}

\newtheorem{ind}[]{{\rm\it Indice}}

\renewcommand{\epsilon}{\varepsilon}

\usepackage{multicol}

\usepackage{tikz}
\usetikzlibrary{arrows}

\usepackage{paralist}
\usepackage{cite}

\title[Cranks for colored partitions]{Cranks for Ramanujan-type congruences of $k$-colored partitions}
\address{Department of Mathematics, Vanderbilt University, Nashville, TN 37240}
\author[Rolen]{Larry Rolen}
\email{larry.rolen@vanderbilt.edu}
\author[Tripp]{Zack Tripp}
\email{zachary.d.tripp@vanderbilt.edu}
\author[Wagner]{Ian Wagner}
\email{ian.c.wagner@vanderbilt.edu}
\address{Department of Mathematics, University of Tennessee, Knoxville, TN 37916}
\author[Wilson]{Samuel Wilson}
\email{swils133@vols.utk.edu}

\usepackage{listings}

\begin{document}
\numberwithin{equation}{section}

\maketitle

{\centering\footnotesize {\it Dedicated to the memory of Freeman Dyson.}\par}

\begin{abstract}
Dyson famously provided combinatorial explanations for Ramanujan's partition congruences modulo $5$ and $7$ via his rank function, and postulated that an invariant explaining all of Ramanujan's congruences modulo $5$, $7$, and $11$ should exist. Garvan and Andrews-Garvan later discovered such an invariant called the crank, fulfilling Dyson's goal. Many further examples of congruences of partition functions are known in the literature. Here, we provide a framework for discovering and proving such invariants for families of congruences. As a first example, we find a family of crank functions that simultaneously explains most known congruences for colored partition functions. The method used, which utilizes Gritsenko, Skoruppa, and Zagier's powerful recent theory of theta blocks, should also be useful for studying other combinatorial functions.
\end{abstract}

\section{Introduction}
Ramanujan famously gave three special congruences for the integer partition function $p(n)$:
 \begin{align*}
 p(5n+4) & \equiv 0 \pmod{5} \\
 p(7n+5) & \equiv 0 \pmod{7} \\
 p(11n+6) & \equiv 0 \pmod{11}.
 \end{align*}
 Ramanujan gave analytic proofs of the first two \cite{Ramanujan}, and Hardy derived proofs of all three from Ramanujan's unpublished work shortly after Ramanujan's passing.
These congruences have inspired a great deal of work by many authors. They are distinguished in that $5, 7,$ and $11$ are the {\it only} primes for which linear congruences of this shape can hold, as postulated by Ramanujan and proven by Ahlgren and Boylan \cite{AhlgrenBoylan}. While many generating function proofs and explanations of these congruences were given over the years, it was not until the work of Dyson that combinatorial explanations were given. Dyson brilliantly defined the {\it rank} of a partition as its largest part minus the number of parts, which can be thought of as a measure of the failure of a partition to be symmetric under conjugation of Ferrers diagrams (or Young Diagrams). He then conjectured that it breaks up partitions of the right size into equinumerous sets modulo $5$ and $7$. For example, modulo $5$ he conjectured there are an equal number of partitions of $5n+4$ with rank congruent to $0, 1, 2, 3,$ or $4$. This was later proven by Atkin and Swinnerton-Dyer \cite{AtkinSD}. The generating function for ranks, and the various properties of them, have had many important applications, and provided early important examples of mock modular forms \cite{BringmannOno,Zagier}. 

However, as Dyson noted, the rank function doesn't explain the Ramanujan congruence modulo $11$. Dyson did, however, conjecture the existence of another function, which he called the {\it crank}, that would interpolate all three Ramanujan congruences simultaneously. Garvan discovered a variant of this function in \cite{Garvan1,Garvan2}, which he and Andrews later used \cite{AndrewsGarvan} to construct Dyson's conjectured crank statistic. We now define the crank as follows. If $\lambda$ is a partition, let $\ell(\lambda)$ be the largest part, and let $\omega(\lambda)$ be the number of $1$'s in the partition. Finally, let $\mu(\lambda)$ be the number of parts of $\lambda$ bigger than $\omega(\lambda)$. Then the crank of $\lambda$ is:
\[
\begin{cases}
\ell(\lambda)& \text{ if } \omega(\lambda)=0,
\\
\mu(\lambda)-\omega(\lambda)& \text{ if } \omega(\lambda)\neq0.
\end{cases}
\]
Although it has a more complicated definition than the rank, Andrews and Garvan showed that it satisfies Dyson's dream of a statistic that splits up partitions in accordance with all three of Ramanujan's congruences. Its generating function is also important, and has the following shape. If $M(m,n)$ is the number of partitions of $n$ with crank $m$ except for $n=1$ where $M(-1,1)=M(0,1)=M(1,1)=1$, then the generating function is
\begin{align}\label{Eqn: crank generating function}
\begin{split}
\mathcal C(z; \tau) := \sum_{m,n} M(m,n) \zeta^m q^n = \prod_{n \geq 1} \frac{1-q^n}{(1- \zeta q^n)(1 -\zeta^{-1} q^n)} = q^{\frac{1}{24}} ( \zeta^{-\frac{1}{2}} - \zeta^{\frac{1}{2}}) \frac{\eta(\tau)^2}{\theta(z;\tau)}.
\end{split}
\end{align}
Here and throughout, we define $q:= e^{2\pi i \tau}$ and $\zeta := e^{2\pi i z}$. We also let $\eta(\tau)$ be the Dedekind eta function
\begin{equation}\label{eta}
\eta(\tau) := q^{1/24} \prod_{n =1}^{\infty} (1-q^n) = \sum_{n =1}^{\infty} \left( \frac{12}{n} \right) q^{n^2 /24}  = \sum_{n \in \Z} (-1)^n q^{\frac{3\left(n-\frac{1}{6} \right)^2}{2}},
\end{equation}
which is a weight $1/2$ modular form with multiplier system we'll denote by $\epsilon$ (for more details, see, for example, Chapter 3 of \cite{Apostol}).  Furthermore, the Jacobi theta function is
\begin{equation}\label{theta}
\theta(z; \tau) := q^{1/8} (\zeta^{1/2} - \zeta^{-1/2}) \prod_{n=1}^{\infty} (1-q^n)(1-\zeta q^n)(1 - \zeta^{-1} q^n) = \sum_{n = - \infty}^{\infty} \left(\frac{-4}{n} \right) \zeta^{n/2} q^{n^2 /8} ,
\end{equation}
and is a Jacobi form of weight $1/2$, index $1/2$ and multiplier system (character) $\epsilon^3$ (for more on Jacobi forms, see the foundational text of Eichler and Zagier \cite{EZ}). 
Thus, the crank generating function is a meromorphic Jacobi form of weight $1/2$ and index $-1/2$. 

We also observe the equidistribution of the appropriate partitions by crank modulo $5$, $7$, and $11$. We provide an equivalent formulation of this result in terms of divisibility by polynomials (see Lemma~\ref{Lem: equidistribution}).
\begin{Thm}[Andrews-Garvan]\label{Thm: crank equidistribution}
For all $n\geq1$ and $\ell\in\{5,7,11\}$, we have the divisibility relation
\[
\Phi_{\ell}(\zeta)\,\bigg | \,\left[q^{\ell n-b_{\ell}}\right]\mathcal C(z;\tau)
\]
as Laurent polynomials,
where $\Phi_{\ell}(X)$ is the $\ell$-th cyclotomic polynomial, $[q^{m}]F(q)$ denotes the $m$-th Fourier coefficient of $F(q)$, and where $b_{\ell}:=(\ell^2-1)/24$.
\end{Thm}

In this paper, we will extend Andrews' and Garvan's crank function to a natural infinite family using Gritsenko, Skoruppa, and Zagier's theory of {\it theta blocks} \cite{GSZ}.
To describe this, we recall that a {\bf{$k$-colored partition}} of a positive integer $n$ is a $k$-tuple of partitions $(\lambda^{(1)}, \dots, \lambda^{(k)})$ such that $\sum_{i=1}^{k} \left| \lambda^{(i)} \right| = n$.  We let $p_{k}(n)$ denote the number of $k$-colored partitions of $n$.  Then the generating function of $p_{k}(n)$ is given by
 \begin{equation}
P_{k}(\tau):= \sum_{n \geq 0} p_{k}(n) q^n = \prod_{n \geq 1} \frac{1}{(1-q^n)^{k}}.
 \end{equation}
 Of course, we have that $p_1(n)$ is simply the standard partition function $p(n)$. 
 
 Higher colored partition functions also satisfy Ramanujan-type congruences (defined as those where the modulus is a prime which matches the common difference of terms in an arithmetic progression).
One can easily prove congruences for $k$-colored partitions whenever $k \equiv 0, \ell-3, \ell -1 \pmod{\ell}$.  The first congruence follows from working modulo $\ell$ while the second two come from the Jacobi triple product and Euler's pentagonal number theorem respectively.  However, there are many more congruences that occur at a ``deeper'' level. In particular, we have the  following result of Boylan \cite{Boylan} and Dawsey and the third author \cite{DW}. 

\noindent For the remainder of the paper, we assume $\ell \geq 5$.

\begin{Thm}[Boylan-Dawsey-Wagner]\label{c}
 Let $k+h=\ell t$ for a prime $\ell$ and positive integers $h$ and $t$ and let $\delta_{k,\ell}$ be such that $24 \delta_{k,\ell} \equiv k \pmod{\ell}$.  
 Then we have the Ramanujan-type congruence 
 \[
 p_k(\ell n +\delta_{k,\ell})\equiv0\pmod\ell
 \]
 if any of the following hold:
 \begin{enumerate}
 \item We have $h \in \{4, 8, 14 \}$ and $\ell \equiv 2 \pmod{3}$.
 \item We have $h \in \{6, 10 \}$ and $\ell \equiv 3 \pmod{4}$.
\item We have $h =26$ and $\ell \equiv 11 \pmod{12}$.
\end{enumerate}
\end{Thm}
\noindent Note that in the above theorem, one could take $\delta_{k,\ell} \in \{0,1,\ldots,\ell-1\}$ such that the above congruence holds, which is always possible as $\ell \geq 5$. \\

Boylan also classifies all Ramanujan-type congruences for all odd $k \leq 47$ and points out that there are a few congruences, which he calls superexceptional congruences, that do not fit into the above families.  However, the congruences from $k \equiv 0, \ell-1, \ell-3 \pmod \ell$ and from Theorem~\ref{c} are believed to cover nearly all Ramanujan-type congruences for colored partitions.  As an example, for the primes $5$ and $7$,
we have the following congruences: 
\begin{equation}\label{ExamplesIntro}
\begin{aligned}
p_{5t}(5n + \delta) & \equiv 0 \pmod{5} \qquad 1 \leq \delta \leq 4, \\
p_{5t+1}(5n+4) &\equiv 0 \pmod{5}, \\
p_{5t+2}(5n+\delta) & \equiv 0 \pmod{5} \qquad \delta = 2,3,4, \\
p_{5t+4}(5n + \delta) &\equiv 0 \pmod{5} \qquad \delta = 3,4, \\
p_{7t}(7n + \delta) & \equiv 0 \pmod{7} \qquad 1 \leq \delta \leq 6, \\
p_{7t +1}(7n+5) &\equiv 0 \pmod{7}, \\
p_{7t+4}(7n +\delta) &\equiv 0 \pmod{7} \qquad \delta = 2,4,5,6, \\
p_{7t+6}(7n + \delta) & \equiv 0 \pmod{7} \qquad \delta = 3,4,6, \\
\end{aligned}
\end{equation}

This paper constructs combinatorial explanations for most of these congruences with respect to crank generating functions. 

 Here, we define a distinguished infinite family of combinatorial statistics for many of the above congruence progressions (see Theorem \ref{MainThm}). These explain many of the congruences for $k$-colored partitions and in a majority of cases explains all of them. In order to obtain our crank generating function, we first define
\begin{equation}\label{general Ck}
\mathcal C_{k}(z_1, z_2, \dots, z_{\frac{k+
\delta_{\text{odd}}(k)}{2}}; \tau) := \mathcal C(0; \tau)^{\frac{k-\delta_{\text{odd}}(k)}{2}} \prod_{i=1}^{\frac{k+\delta_{\text{odd}}(k)}{2}}\mathcal C(z_{i}; \tau),
\end{equation}
where $\delta_{\text{odd}}(k):=(1-(-1)^{k})/2$ is the indicator function for $k$ being odd. We will be interested in the case where our $z_i$'s are specialized to $z_i = a_i z$ for integers $a_{1} > a_{2} > \cdots > a_{\frac{k+\delta_{\text{odd}}(k)}{2}} >0$ and $z\in\C$. Explicitly, the $k$-colored crank generating functions will be of the form
\[
\left(\frac{q^{\frac{1}{24}}}{\eta(\tau)}\right)^{\frac{k-\delta_{\text{odd}}(k)}{2}} \times \prod_{i=1}^{\frac{k+\delta_{\text{odd}}(k)}{2}}\mathcal C(a_{i}z; \tau).
\]
Expanding this using the product formula for the crank generating function given in \eqref{Eqn: crank generating function}, we have the equivalent form
\begin{equation}\label{crankprod}
 \prod_{n \geq 1} \frac{1-\delta_{\text{odd}}(k)q^n}{(1-\zeta^{\pm a_{1}} q^n)(1- \zeta^{\pm a_2} q^n) \cdots (1-\zeta^{\pm a_{\frac{k+\delta_{\text{odd}}(k)}{2}}} q^n )},
\end{equation}
where $(1-\zeta^{\pm a} q^n) = (1-\zeta^a q^n)(1- \zeta^{-a} q^n)$.  These crank functions can be interpreted combinatorially as weighted sums of cranks of each color in the $k$-colored partition. In particular, this $k$-colored crank counts the crank of the partition with the first color weighted by $a_{1}$ plus the crank of partition with the second color weighted by $a_{2}$ and so on until the $\frac{k+1}{2}$-th color for $k$ odd and $\frac{k}{2}$-th color for $k$ even.  If $k + 14$ is not divisible by any prime $\ell \equiv 2 \pmod{3}$, then define 
\begin{equation}\label{first def}
\mathcal C_{k}(z;\tau) :=\mathcal C_{k}(kz, (k-2)z, \dots, (2-\delta_{\text{odd}}(k))z; \tau).
\end{equation} 
If $k + 14$ is divisible by a prime $\ell \equiv 2 \pmod{3}$, then define
\begin{equation}\label{second def}
\mathcal C_{k}(z;\tau) := \mathcal C_{k}((k+2)z, (k-2)z, \dots, (2-\delta_{\text{odd}}(k))z; \tau).
\end{equation}
Note here, we are using a specific choice for the $a_i$. Indeed, there are many choices that may work, however this choice seems to be the simplest in terms of calculations. Our main result gives many situations in which these crank functions combinatorially interpolate the Ramanujan-type congruences for colored partitions.

\begin{Thm} \label{MainThm}
Let $\{\ell n+\delta_{k,\ell}\}_{n\in\Z_{\geq0}}$ be an arithmetic progression where there is a congruence for $p_k(n)$ modulo $\ell$ afforded by Theorem~\ref{c} such that
\begin{enumerate}
 \item[i).] There is no odd prime $p \equiv 2 \pmod{3}$ such that $p|(k+14)$ or we either have $\ell \not\equiv 2 \pmod{3}$ or $\ell\nmid(k+8)$.
\item[ii).] We have $\ell\not\equiv11\pmod{12}$ or $\ell\nmid(k+26)$.
\end{enumerate}
 For $n\geq0$ we have the divisibility relation
\[
\Phi_{\ell}(\zeta)\,\bigg | \,\left[q^{\ell n+\delta_{k,\ell}}\right]\mathcal C_{k}(z;\tau).
\]
\end{Thm}
We will refer to the above congruences as \textit{generic}. Otherwise, we say they are \textit{singular}. Our choice of the word ``singular" here is simply to indicate that these are the congruences that are unexplained by our cranks.

\begin{rmk}
Some analogues of crank functions have been defined before.  For two colors Hammond and Lewis define a statistic called the birank that explains the Ramanujan-type congruences \cite{HL} and Andrews defines a bicrank statistic that explains one such congruence \cite{Andrews}.  In \cite{Garvan3} Garvan presents other statistics for two colors and extends both the Hammond-Lewis birank and Andrews' bicrank to a multirank and multicrank which he then proves explain the Ramanujan-like congruences when $k \equiv -1, -3 \pmod{\ell}$ and $k$ is even. Garvan's functions are related to our functions in these special cases.
\end{rmk}
\begin{rmk}
The definition given in  \eqref{first def} will always explain a congruence coming from $h=8$.  However, if there is a congruence coming from $h=14$ as well our definition can't explain both simultaneously.
\end{rmk}
\begin{rmk}
The proof of this result relies on the repackaging of the Macdonald identities in \cite{MacDonald}.  It is currently still an open problem to find a Macdonald-type identity for the $26$th power of $\eta(\tau)$ which keeps us from proving this result in the $h=26$ case. 
\end{rmk}
\begin{rmk}
At first glance it may appear that Theorem \ref{MainThm} suggests that there is no crank function that describes all congruences when there are, for example, congruences coming from both $h=14$ and $h=8$.  This is not necessarily true.  When the primes in the moduli of some of the congruences are small enough crank functions can be defined that explain these congruences in different ways than the method used in this paper.  For example, there are congruences coming from $9 \equiv -8 \pmod{17}$ and $9 \equiv -14 \pmod{23}$, and the crank function
\begin{equation*}
\mathcal C_{9}(9z,7z,6z,4z,z; \tau)
\end{equation*}
explains all of the $9$-colored partition congruences as shown in the example in Section \ref{k=9}. However, such examples seem to be more sporadic and don't naturally fit into the same family. For instance, of the congruences above modulo $5$ and $7$ displayed in \eqref{ExamplesIntro}, $\ell=7$ can't have any singular cases as it is not congruent to $2$ modulo $3$. For $\ell=5$, we have that $\ell\not\equiv3\pmod 4$, and we can't have $5|(5t+j+8)$ unless $j=2$ (coming from the first condition for a generic congruence). Thus, in $21$ of the $24$ above families of congruences, all progressions are automatically generic. In the remaining $3$ cases, the congruence is singular if and only if $5t+16$ is divisible by an odd prime $p\equiv2\pmod 3$.

\end{rmk}

The remainder of the paper is organized as follows. In Section~\ref{Prelim}, we describe the theory of theta blocks. We continue in Section~\ref{wt1} with the description of the specific theta blocks we require, and we note some of their key properties. The proofs of the main results are in Section~\ref{MainThmSec}. We conclude in Section~\ref{FinalSection} with some illustrative examples and discussions. 

We end this section with some relevant questions.

\subsection{Questions for further study}
\begin{enumerate}
\item How can one find crank functions to explain all of the singular congruences? Can a family of cranks that simultaneously ``explains''
all colored partition congruences be found? Similarly, can one explain the superexceptional congruences of Boylan?
\item Can one use these methods to find further families of crank functions in combinatorics?
\item What are the combinatorial, congruence, or analytic properties of these new crank functions (as have been extensively studied for 
ranks and cranks of partitions)?
\item Similar to the Kac-Wakimoto interpretation of the powers of the crank generating function as in \cite{BringmannCreutzigRolen}, is there an interesting representation-theoretic interpretation
of the crank functions here?
\item Is there interesting number-theoretic content to the $\ell$-dissections of these crank functions (see, e.g., \cite{AndrewsBerndtYourTopHitParade}). 
\item Is it possible to choose the $a_i$ in a consistent way in order to satisfy a version of Stanton's conjecture as stated in \cite{BMRStanton}. 
\end{enumerate}

\section*{Acknowledgements}
The authors thank Amanda Folsom, Michael Mertens, Martin Raum, Ae Ja Yee, and the anonymous referees for useful comments and corrections which improved the paper. This work was supported by a grant from the Simons Foundation (853830, LR).

\section{Preliminaries on theta blocks}\label{Prelim}

As mentioned above, the main theoretical tool behind our results is a preprint on theta blocks by Gritsenko, Skoruppa, and Zagier \cite{GSZ}. Here, we give a detailed overview of the facets of this theory that we will require. This discussion follows more or less directly from \cite{GSZ}, however, the authors thought this discussion would be useful to the reader to make the paper more self-contained. For the purposes of our proof of Theorem~\ref{MainThm}, it is logically possible to skip this section since the identities arising from Theorem~\ref{R2} are previously known. However, as we will see in Table~\ref{TableThetaBlocks}, the theory of theta blocks is what informs our choice of $q$-series identity used to prove the equidistribution of our crank. Without this theory the authors would not have been able to identify the crank functions defined in this paper. Moreover, our use of theta blocks points to a large theory that can help explain other well-known congruences. See Section~\ref{examples} for details and some examples.

We first discuss the notion of Jacobi forms, two variable functions which have a mixture of the properties of modular forms and elliptic functions. To do so, we first recall some facts about integral lattices.

Suppose that $\underline{L}=(L,\beta)$ is an integral lattice of rank $n$ where $\beta\colon L \times L \to \Z$ is a symmetric non-degenerate bilinear form.  If $U$ is a $\Z$-submodule of full rank in $\Q \otimes L$ denote the dual subgroup by $U^{\#} := \{ y \in \Q \otimes L: \beta(y,x) \in \Z \ \forall x \in U \}$.  Let $\beta(x) := \frac{1}{2} \beta(x,x)$. Note that $\beta(x)$ is not necessarily integral.  If it is, then $\underline{L}$ is called {\it{even}} and is otherwise called {\it{odd}}.  In any case, the map $x \mapsto \beta(x)$ defines an element of order at most $2$ in the dual group ${\rm{Hom}}(L, \Q/\Z)$ of $L$ that is trivial on $2L$.  The kernel $L_{ev}$ of this homomorphism defines an even sublattice of index at most $2$ in $L$.  Since $\beta$ is non-degenerate, there exists an element $r \in \Q \otimes L$ such that $\beta(x) \equiv \beta(r,x) \pmod{\Z}$ for $x \in L$.  The {\it{shadow}} of $L$ is defined by
\begin{equation*}
L^{\bullet} := \{ r \in \Q \otimes L: \beta(x) \equiv \beta(r,x) \pmod{\Z} \ \text{for all} \ x \in L \}.
\end{equation*}    
If $L$ is even, then $L^{\bullet} = L^{\#}$ and if $L$ is odd, then $L_{ev}^{\#} = L^{\#} \cup L^{\bullet}$.

We may now define Jacobi forms.

\begin{defn}
For an integral lattice $\underline{L}=(L,\beta)$ equipped with a symmetric non-degenerate bilinear form $\beta$, a {\bf{Jacobi form of weight $k$, lattice index $\underline{L}$ and character $\epsilon^h$}} is a holomorphic function $\phi(\bz; \tau)$ with $\tau \in \HH$ and $\bz \in \C \otimes L$ which satisfies the following properties:
\begin{enumerate}
\item For all $A = \begin{pmatrix} a & b \\ c & d \end{pmatrix} \in \SL_{2}(\Z)$, we have the modularity transformation
\[
\phi \left( \frac{\bz}{c \tau +d}; A \tau \right) = e \left( \frac{c \beta(\bz)}{c \tau +d} \right) (c \tau + d)^{k} \epsilon^{h}(A) \phi(\bz; \tau).
\]
\item For all $x, y \in L$, we have the elliptic transformation
\[
\phi(\bz+x \tau + y; \tau) = e(\beta(x+y) - \tau \beta(x) - \beta(x,\bz)) \phi(\bz; \tau).
\]
\item The Fourier expansion of $\phi$ is of the form
\[
\phi(\bz; \tau) = \sum_{n \in \frac{h}{24} + \Z} \sum_{\substack{r \in L^{\bullet} \\ n \geq \beta(r)}} c(n,r) e(\beta(r,\bz)) q^n.
\]
\end{enumerate}
\noindent If $\phi(\bz;\tau)$ is a Jacobi form, then $\phi(\bz;\tau)/\eta(\tau)$ is again a Jacobi form, provided that condition (3) is satisfied.
\end{defn}

\begin{rmk}
    Note that the $k$ in the above definition is unrelated to the number of colors in a partition, as discussed in the introduction.
\end{rmk}

The space of Jacobi forms of weight $k$, index $\underline{L}$ and character $\epsilon^h$ is denoted by $J_{k, \underline{L}}(\epsilon^h)$. Additionally, for an integer $a$ and a Jacobi form $\phi(\bz; \tau)$ we will let $\phi_{a}$ denote the rescaled function $\phi(a\bz; \tau)$. We refer the reader to Eichler-Zagier's foundational text \cite{EZ} for more details on their general theory and to \cite{S} for more details on lattice index Jacobi forms. 

Let $k$ and $h$ be rational numbers such that $k \equiv \frac{h}{2} \pmod{\Z}$ and let $\epsilon^h$ be the character of $\eta(\tau)^h$. Suppose that $\alpha : \underline{L} \to \underline{M}$ is an isometric embedding. Then the application \[\alpha^{*} \phi(\bz; \tau) := \phi (\alpha \bz; \tau)\] defines a map
\[
\alpha^{*} : J_{k, \underline{M}}(\epsilon^h) \to J_{k, \underline{L}}(\epsilon^h).
\]
There are two embeddings of particular interest.  The first is when a lattice $\underline{L}$ can be isometrically embedded into $\underline{\Z}^{N} := (\Z^{N}, \cdot)$ where $\cdot$ denotes the standard scalar product.  If $\alpha_{j}$ are the coordinate functions of this embedding so that $\beta(x,x) = \sum_{j} \alpha_{j}(x)^{2}$, then
\begin{equation*}
\prod_{j=1}^{N} \theta( \alpha_{j}(\bz); \tau) \in J_{\frac{N}{2}, \underline{L}}( \epsilon^{3N}).
\end{equation*}
Conversely, if such a product defines a Jacobi form of index $\underline{L}$, then we must have $\beta(x,x) = \sum_{j} \alpha_{j}(x)^2$ and the $\alpha_{j}$ define an isometric embedding of $\underline{L}$ into $\underline{\Z}^{N}$. 

The other interesting embedding is of the form
\begin{align*}
s_{x}: \left( \Z, (u,v) \mapsto muv \right) & \to \underline{L} \\
s_{x}(u) & = ux 
\end{align*}
where $x$ is a non-zero element in $L$ with $m= \beta(x,x)$.  For any $w\in\C$, this gives a map
\begin{align*}
s_{x}^{*}: J_{k, \underline{L}}(\epsilon^h) & \to J_{k, \frac{m}{2}}(\epsilon^h) \\
\phi(\bz; \tau) & \mapsto \phi(x w; \tau),
\end{align*}
where $J_{k, \frac{m}{2}}(\epsilon^h)$ is a space of scalar-indexed Jacobi Forms described in \cite{EZ}.
  
Another key concept is that of a {\it{eutactic star of rank $N$}} on a lattice $\underline{L}=(L, \beta)$, defined as a family $s$ of non-zero vectors $s_{j} \in L^{\#}$ ($1 \leq j \leq N$) such that
\begin{equation*}
x = \sum_{j=1}^{N} \beta(s_{j}, x)s_{j}
\end{equation*}
for all $x \in \Q \otimes L$.  For a eutactic star $s$, one has
\begin{equation*}
\beta(x,x) = \sum_{j} \beta(s_{j}, x)^2
\end{equation*}
for all $x$, so the map $x \mapsto \left( \beta(s_{j}, x) \right)_{1 \leq j \leq N}$ defines an isometric embedding $\alpha_{s}: \underline{L} \to \underline{\Z}^N$.  Conversely, if $\alpha$ is such an embedding then there exist vectors $s_{j}$ such that the $j$th coordinate function of $\alpha$ is given by $\beta(s_{j}, x)$ and the family $s_{j}$ is a eutactic star.

Now let $G$ be a subgroup of the orthogonal group $O(\underline{L})$ that leaves $s$ invariant up to signs, i.e., for each $g \in G$ there exists a permutation $\sigma$ of the indices $1 \leq j \leq N$ and signs $\varepsilon_{j} \in \{ \pm 1 \}$ such that $gs_{j} = \varepsilon_{j} s_{\sigma(j)}$ for all $j$.  We set
\begin{equation*}
{\rm{sn}}(g) := \prod_{j} \varepsilon_{j}.
\end{equation*}
It follows that $g \mapsto {\rm{sn}}(g)$ defines a linear character ${\rm{sn}}\colon G \longrightarrow \{ \pm 1 \}$.  The group $G$ acts naturally on $L^{\bullet}/L_{ev}$.  We call the eutactic star $s$ {\it{$G$-extremal on $\underline{L}$}} if there is exactly one $G$-orbit in $L^{\bullet}/L_{ev}$ whose elements have their stabilizers in the kernel of ${\rm{sn}}$. The authors of \cite{GSZ} then proved the following.
\begin{Thm} \label{L}
Let $\underline{L}=(L, \beta)$ be an integral lattice of rank $n$, let $s$ be a $G$-extremal eutactic star of rank $N$ on $\underline{L}$.  Then there is a constant $\gamma$ and a vector $w \in L^{\bullet}$, dependent $g$ and $\epsilon_j$, such that
\begin{equation}
\eta(\tau)^{n-N} \prod_{j=1}^{N} \theta(\beta(s_{j}, \bz); \tau) = \gamma \sum_{x \in w + L_{ev}} q^{\beta(x)} \sum_{g \in G} {\rm{sn}}(g) e(\beta(gx, \bz)).
\end{equation}
In particular, the product on the left defines an element of $J_{\frac{n}{2}, \underline{L}}(\epsilon^{n+2N})$.
\end{Thm}
Note that we can always choose $w = \frac{1}{2} \sum_{j} \varepsilon_{j} s_{j}$, where $\epsilon_j = 1$ if $gs_j=s_{\sigma(j)}$ and $\epsilon_j = -1$ if $gs_j=-s_{\sigma(j)}$. This convenient choice of $w$ is a consequence of our use of a eutactic star, as it essentially takes the sum of positive roots from a root system (see Section 11 in \cite{GSZ}).
Many of the interesting examples of Theorem \ref{L} come from lattices constructed from root systems.  The following theorem of \cite{GSZ} shows how to construct a theta block from a root system.
\begin{Thm} \label{R1}
Let $R$ be a root system of dimension $n$, let $R^{+}$ be a system of positive roots of $R$ and let $F$ denote the subset of simple roots in $R^{+}$.  For $r \in R^+$ and $f \in F$, let $\gamma_{r,f}$ be the (non-negative) integers such that $r = \sum_{f \in F} \gamma_{r,f} f$.  The function
\begin{equation*}
\theta_{R}(\bz; \tau) := \eta(\tau)^{n - |R^+ |} \prod_{r \in R^+} \theta \left( \sum_{f \in F} \gamma_{r,f} z_{f} ; \tau \right)
\end{equation*}
($\tau \in \HH, \bz = \{z_{f} \}_{f \in F} \in \C^{F}$) defines a Jacobi form in $J_{\frac{n}{2}, \underline{R}}(\epsilon^{n +2N})$, where the lattice $\underline{R}$ is $\Z^{F}$ equipped with the quadratic form $Q(\bz) := \frac{1}{2} \sum_{r \in R^+} \left(\sum_{f} \gamma_{r,f} z_{f} \right)^2$.
\end{Thm}
Computing the sum side of a theta block can be simplified when the lattice comes from a root system. This computation requires a few preliminary definitions.  Let $R$ be an irreducible root system of dimension $n$, let $N = |R^+ |$, and let $E_{R}$ be the ambient Euclidean space of $R$ with inner product $( \cdot, \cdot)$.  Define the number 
\begin{equation*}
\mathfrak{h} := \frac{1}{n} \sum_{r \in R^+} (r,r),
\end{equation*}
let $W_{R}$ be the lattice
\begin{equation*}
W_{R} := \left\{ x \in E_{R}: \frac{(x,r)}{\mathfrak{h}} \in \Z \ \text{for all} \ r \in R \right\},
\end{equation*}
and set
\begin{equation*}
\underline{R} := \left(W_{R}, \frac{(\cdot, \cdot)}{\mathfrak{h}} \right).
\end{equation*}
The dual lattice $W_{R}^{\#}$ equals the lattice spanned by the roots $r \in R$.  Let $G_{R}$ be the Weyl group of $R$.  Then it can be shown that the eutactic star $R^+$ on $\underline{R}$ is extremal with respect to the Weyl group $G_{R}$.  The Weyl group is generated by the reflections about the hyperplane perpendicular to a root $r \in R$ for each $r$.  Explicitly, it is generated by all the $s_{r}$, given by
\[
s_{r}(v) = v - 2 \frac{(r,v)}{(r,r)} r,
\]
where $(\cdot, \cdot)$ is the inner product of $E_{R}$. Following \cite{GSZ}, we can now state the equivalent of Theorem \ref{L} specifically for root systems.
\begin{Thm} \label{R2}
Let $R$ be an irreducible root system with a choice of positive roots $R^+$, and let $w = \frac{1}{2} \sum_{r \in R^+} r$.  Then we have
\[
\theta_{R}(\bz; \tau) := \eta(\tau)^{n-N} \prod_{r \in R^+} \theta \left( \frac{(r,\bz)}{\mathfrak{h}}; \tau \right) = \sum_{x \in w + W_{R,ev}} q^{\frac{(x,x)}{2h}} \sum_{g \in G_{R}} {\rm{sn}}(g) e \left( \frac{(gx, \bz)}{\mathfrak{h}} \right)
\]
for all $\tau \in \HH$ and $z \in \C \otimes W_{R}$.  In particular, $\theta_{R}$ is a holomorphic Jacobi form in $J_{\frac{n}{2}, \underline{R}}(\epsilon^{n + 2N})$.
\end{Thm}
Note that for any $f \in F$ and $\bz \in \C \otimes W_{R}$ we can set $z_{f} := \frac{(f,\bz)}{\mathfrak{h}}$ and the application $\bz \mapsto \{z_{f} \}_{f \in F}$ defines an isomorphism of $\C$-vector spaces $\C \otimes W_{R} \to \C^{F}$ which maps $W_{R}$ onto $\Z^{F}$.  From this we see that $\frac{(r,\bz)}{\mathfrak{h}} = \sum_{f \in F} \gamma_{r,f} z_{f}$ and $\frac{(x,x)}{2\mathfrak{h}} = Q(\{x_{f}\})$ with $Q$ as in Theorem \ref{R1}.  This shows that Theorem \ref{R1} is a weaker form of this theorem.
Also note that the identities given in Theorem \ref{R2} are a repackaging of the Macdonald identities of \cite{MacDonald}, as noted in \cite{GSZ}.   The proof in \cite{GSZ} gives an alternative proof of the Macdonald identities. 

\section{Weight $1$ theta blocks}\label{wt1}
In this section we will explicitly compute some of the sum-to-product formulas given by Theorem \ref{R2} in order to illustrate that certain coefficients of the theta blocks we are studying vanish. In Section~\ref{MainThmSec}, this will allow us to prove that these coefficients are divisible by $\Phi_\ell(\zeta)$ by the following lemma.

\begin{Lemma}\label{Lem: equidistribution}
	Let $f(\zeta)$ be a Laurent polynomial and $\ell$ a prime. Define \\ 
	\noindent $F(r,\ell) = \sum_{j \equiv r \pmod{\ell}} [\zeta^j]f(\zeta)$. Then $\Phi_\ell(\zeta)\mid f(\zeta)$ in $\Q[\zeta^{-1}, \zeta]$ if and only if
	\begin{equation*}
		F(0,\ell) = F(1,\ell) = \ldots = F(\ell -1, \ell).
	\end{equation*} 
\end{Lemma}

\begin{proof}
	Multiplying $f(\zeta)$ by a sufficiently large power of $\zeta$ divisible by $\ell$ and using the fact that $\gcd(\zeta, \Phi_\ell(\zeta)) = 1$, we may reduce to the case that $f(\zeta) \in \Q[\zeta]$. Since $\Phi_\ell(\zeta)$ is irreducible over $\Q[\zeta]$, it is a standard fact from algebra that $\Phi_\ell(\zeta)\mid f(\zeta)$ is equivalent to $f(\zeta_\ell) = 0$, where $\zeta_\ell$ the primitive $\ell$-th root of unity. Writing $f(\zeta) = \sum a_j \zeta^j$, we see that
	\begin{align}\label{Eqn: root of unity evaluation}
		f(\zeta_\ell) = \sum a_j \zeta_\ell^j &= \sum_{r=0}^{\ell -1} \left( \sum_{j \equiv r\pmod{\ell}} a_j \right) \zeta_\ell^r = \sum_{r = 0}^{\ell - 1} F(r,\ell) \zeta_\ell^r \nonumber\\
		&= \sum_{r=0}^{\ell - 2} (F(r,\ell) - F(\ell -1, \ell))\zeta_\ell^r,
	\end{align}
	where in the last equality we have used the well-known fact that $1 + \zeta_\ell + \dots + \zeta_{\ell}^{\ell - 1} = 0$. Recall by the primitive element theorem that $1, \zeta, \dots, \zeta^{\ell -2}$ is a basis for $\Q[\zeta]$ over $\Q$, so we see that \eqref{Eqn: root of unity evaluation} is $0$ if and only if $F(r,\ell) = F(\ell - 1, \ell)$ for $r = 0, \dots, \ell - 2$. 
\end{proof}

\noindent By taking $f(\zeta) = \sum_m M(m,n)\zeta^m$, Lemma~\ref{Lem: equidistribution} proves that Theorem~\ref{Thm: crank equidistribution} is simply restating the well-known result on equidistribution of the crank.

Now, define
\begin{align*}
\theta^{*}(z; \tau) &= \eta(\tau) \frac{\theta(2z; \tau)}{\theta(z; \tau)} = q^{1/24} (\zeta^{1/2} + \zeta^{-1/2}) \prod_{n=1}^{\infty} (1-q^n)(1+ \zeta^{\pm 1} q^n)(1 - \zeta^{\pm 2} q^{2n-1}) \\
&= \sum_{n=-\infty}^{\infty} \left(\frac{12}{n} \right) \zeta^{n/2} q^{n^2 /24} = \sum_{n = -\infty}^{\infty} (-1)^n q^{\frac{(6n+1)^2}{24}} \zeta^{3n+\frac{1}{2}} \left[ 1 + \zeta^{-6n-1} \right],
\end{align*}
which is a Jacobi form of weight $1/2$, index $3/2$ and multiplier system $\epsilon$.  The theta blocks we are interested in were singled out in Theorem 12.2 of \cite{GSZ} which states that $J_{1,m}(\epsilon^h)$ for $h=4,6,8,10$, and $14$ is spanned by $\phi_{R}(\ell z ; \tau)$ where $\ell$ runs through all elements of $\underline{R}$ with square length $2m$ and $R$ is given in Table~\ref{TableThetaBlocks}.  For $h=2$ the space $J_{1,m}(\epsilon^2)$ contains the theta blocks, but is not necessarily spanned by them.  Further, for all other values of $h \pmod{24}$ one has $J_{1,m}(\epsilon^h) =0$ (recall that $\epsilon$ has order $24$). We are now able to state the main result of Section~\ref{wt1}.

\begin{center}
\begin{table}[h]
\begin{tabular}{ c c c }
\hline \\
 $h$ & $R$ & $\phi_{R}(\bz; \tau)$ \\ 
 \hline \\
 $2$ & $A_{1} \oplus A_{1}$ & $\theta^{*}(z_{1}; \tau) \theta^{*}(z_{2}; \tau)$ \\  
 \hline \\
 $4$ & $A_{1} \oplus A_{1}$ & $\theta(z_{1}; \tau) \theta^{*}(z_{2}; \tau)$ \\ 
 \hline \\
 $6$ & $A_{1} \oplus A_{1}$ & $\theta(z_{1}; \tau) \theta(z_{2}; \tau)$ \\ 
 \hline \\
 $8$ & $A_{2}$ & $ \eta(\tau)^{-1} \theta(z_{1}; \tau) \theta(z_{2}; \tau) \theta(z_1 + z_2; \tau)$ \\ 
 \hline \\
 $10$ & $B_{2}$ & $\eta(\tau)^{-2}\theta(z_{1}; \tau) \theta(z_{2}; \tau) \theta(z_{1} + z_{2}; \tau) \theta(z_{1} + 2z_{2}; \tau)$ \\ 
 \hline \\
 $14$ & $G_{2}$ & $\eta(\tau)^{-4} \theta(z_{1}; \tau) \theta(z_{2}; \tau) \theta(z_{1} + z_{2}; \tau) \theta(2z_{1} + z_{2}; \tau) \theta(3z_{1} + z_{2}; \tau) \theta(3z_{1} + 2z_{2}; \tau)$ \\ 
 \hline
\end{tabular}
\hspace{.5cm}
\caption{Theta blocks $\phi_R$ of character $\epsilon^h$ determined by root systems $R$}
\label{TableThetaBlocks}
\end{table}
\end{center}

\begin{Lemma}\label{vanishingcoeffs}
	Suppose $\{\ell n+\delta_{k,\ell}\}_{n\in\Z_{\geq0}}$ is a generic congruence progression for $P_k(\tau)$ with $k+h = \ell t$ and $h \in \{4,6,8,10,14\}$. Then if $\phi_R(\bz;\tau)$ is the theta block associated to $h$ in Table~\ref{TableThetaBlocks}, then for $a, b \in \Z$ we have
	\begin{equation*}
		\Phi_{\ell}(\zeta) \bigg | \left[q^{\ell n + \delta_{k, \ell} + \frac{h}{24}} \right] \phi_R(az, bz;\tau),
	\end{equation*}
	
\end{Lemma}

\noindent The remainder of the section is devoted to the proof of Lemma~\ref{vanishingcoeffs}. The proof is broken up into cases, and each subsection is another case based on the value of $h$. Only the cases $h=4, 8, 14$ are presented as the other cases can be worked out analogously. Throughout the section, $\bz = (z_1, z_2)$. If a Macdonald-type identity is found for $h=26$ the authors believe that similar techniques will be able to be used to define crank functions explaining the corresponding congruences.

For the remainder of the paper, define the $q$-Pochhammer symbol as
$$
(a)_\infty := (a;q)_\infty := \prod_{n\geq 1}(1-aq^n).
$$
Additionally, define
$$
(a^{\pm b})_\infty := (a^{\pm b};q)_\infty := (a^b;q)_\infty(a^{-b};q)_\infty.
$$

\subsection{$h=4$}
When $k \equiv - 4 \pmod{\ell}$ and $\ell \equiv 2 \pmod{3}$ is prime, then we have $p_{k} \left( \ell n + \frac{\ell^2 -1}{6} \right) \equiv 0 \pmod{\ell}$.  The $h=4$ case is slightly different from the others because there is a $\theta^{*}$ which causes there to be a product in the denominator. In order to control which factors cancel under the specialization $\zeta = \zeta_\ell$, we multiply by $\frac{\theta(az;\tau)\theta^*(bz;\tau)}{\theta(az;\tau)\theta^*(bz;\tau)}$. In this particular case, we will actually want to cancel terms such that the denominator is removed, so we end up with
\begin{equation*}
\widetilde{\theta}_{A_{1} \oplus A_{1}}(az; \tau) = (q;q)_\infty^2(\zeta^{\pm a};q)_\infty
\end{equation*}
in the numerator (this is the Jacobi form $\eta(\tau) \theta(az; \tau)$ with the prefactors removed).  We have
\begin{align*}
\widetilde{\theta}_{A_{1} \oplus A_{1}}(az; \tau) &=\sum_{n \in \Z} (-1)^n q^{\frac{3 \left( n - \frac{1}{6} \right)^2}{2} - \frac{1}{24}} \cdot \sum_{m \in \Z} (-1)^m q^{\frac{\left(m + \frac{1}{2} \right)^2}{2} -\frac{1}{8}} \zeta^{am} \left[ \frac{1 - \zeta^{-a(2m+1)}}{1 - \zeta^{-a}} \right].
\end{align*}
We want to see when the exponent of $q$ here is $\frac{\ell^2 -1}{6} \pmod{\ell}$.  Note that throughout we will write $\frac{1}{t} \pmod{\ell}$ as the multiplicative inverse of $t$ modulo $\ell$.  The power of $q$ is $\frac{\ell^2 -1}{6} \pmod{\ell}$ when $3 \left(n - \frac{1}{6} \right)^2 + \left(m +\frac{1}{2} \right)^2 \equiv 0 \pmod{\ell}$.  A standard quadratic reciprocity argument shows that $-3$ is not a square modulo $\ell$ for a prime $\ell \equiv 2 \pmod{3}$ and so there are no non-zero solutions to $3x^2 + y^2 \equiv 0 \pmod{\ell}$. 
Therefore, the only way we can get the correct residue class is if both of these terms are $0 \pmod{\ell}$ or equivalently $n \equiv \frac{\ell +1}{6} \pmod{\ell}$ and $m \equiv \frac{\ell -1}{2} \pmod{\ell}$.  When $m \equiv \frac{\ell -1}{2} \pmod{\ell}$ and we specialize $\zeta= \zeta_{\ell}$, the term $1-\zeta^{-a(2m+1)}$ is zero and so the coefficients $[q^{\ell n + \frac{\ell^2 -1}{6}}]\widetilde{\theta}_{A_{1} \oplus A_{1}}$ vanish.

\subsection{$h=8$}
When $k\equiv -8 \pmod{\ell}$ and $\ell \equiv 2 \pmod{3}$ we have $p_{k} \left( \ell n + \frac{\ell^2 -1}{3} \right) \equiv 0 \pmod{\ell}$.  Note that in this case (and the next), $x_1$ and $x_2$ are analogous to $a$ and $b$.

In this case the theta block comes from the root system
\begin{equation*}
A_{2} = \{ \pm(1,-1,0), \pm (0,1,-1), \pm(1,0,-1) \},
\end{equation*}
which has associated Euclidean space
\begin{equation*}
E_{A_{2}} = \{x=(x_{1}, x_{2}, x_{3}) \in \R^3 : x_{1} + x_{2} + x_{3} =0\}.
\end{equation*}
We can choose
\begin{equation*}
A_{2}^{+} = \{ (1,-1,0), (0,1,-1), (1,0,-1) \} = \{r_{1}, r_2, r_3 = r_{1} +r_{2} \},\quad F_{A_{2}} = \{r_{1}, r_{2} \}.
\end{equation*}
We have that
\begin{equation*}
G_{A_{2}} = S_{3} = \{Id, (12), (13), (23), (123), (132)\}
\end{equation*}
where elements act by permuting the components of $x \in E_{A_{2}}$ as usual.  Therefore we have
\begin{align*}
{\rm{sn}}(Id) &= {\rm{sn}}((123)) = {\rm{sn}}((132)) =1 \\
{\rm{sn}}((12)) &= {\rm{sn}}((13)) = {\rm{sn}}((23)) =-1.
\end{align*}  
A short computation tells us that $\mathfrak{h}=3$ and $w =(1,0,-1)$.  We also find that
\begin{equation*}
W_{A_{2}} = \left\{ x \in E_{A_{2}} : \frac{(x,r)}{3} \in \Z \ \forall r \in A_{2} \right\} = \{(x_{1}, x_{2}, -x_{1}-x_{2}) \in \Z^3 : x_{1} \equiv x_{2} \pmod{3} \}
\end{equation*}
and $W_{A_{2}, ev} = W_{A_{2}}$.  In general, $x=(x_{1}, x_{2}, -x_{1} - x_{2}) = x_{1}r_{1} + (x_{1}+x_{2})r_{2}$ so we have
\begin{align*}
Id(x) &= x_{1}r_{1} +(x_{1}+x_{2})r_{2} \\
(12)(x) &= x_{2}r_{1} + (x_{1}+x_{2})r_{2} \\
(13)(x) &= -(x_{1}+x_{2})r_{1} - x_{1} r_{2} \\
(23)(x) &= x_{1}r_{1} -x_{2}r_{2} \\
(123)(x) &= -(x_{1}+x_{2})r_{1} -x_{2}r_{2} \\
(132)(x) &= x_{2}r_{1} -x_{1}r_{2}.
\end{align*}
As before, we will make the change of variables $z_{1} = \frac{(r_{1},\bz)}{3}$ and $z_{2}=\frac{(r_{2},\bz)}{3}$ to find
\begin{align*}
\theta_{A_{2}}(\bz; \tau) &= \eta(\tau)^{-1} \theta(z_{1}; \tau) \theta(z_{2}; \tau) \theta(z_{1} + z_{2}; \tau) \\
&= q^{\frac{1}{3}} \zeta_{1}^{-1} \zeta_{2}^{-1}(1-\zeta_{1})(1-\zeta_{2})(1-\zeta_{1}\zeta_{2})
\\
&\times (q)_\infty^2(\zeta_1^{\pm 1})_\infty(\zeta_2^{\pm 1})_\infty((\zeta_1\zeta_2)^{\pm 1})_\infty \\
&=\sum_{\substack{x_{1} ,x_{2} \in  \Z \\ x_{1} +1 \equiv x_{2} \pmod{3}}} q^{\frac{x_{1}^2 + x_{1}x_{2} + x_{2}^2}{3}} \\
&\times \left[ \zeta_{1}^{x_{1}} \zeta_{2}^{x_{1}+x_{2}} + \zeta_{1}^{-x_{1}-x_{2}} \zeta_{2}^{-x_{2}} + \zeta_{1}^{x_{2}} \zeta_{2}^{-x_{1}} - \zeta_{1}^{x_{2}} \zeta_{2}^{x_{1}+x_{2}} - \zeta_{1}^{-x_{1}-x_{2}} \zeta_{2}^{-x_{1}} -\zeta_{1}^{x_{1}} \zeta_{2}^{-x_{2}} \right]
\end{align*}
We therefore find
\begin{align*}
& (q)_\infty^2(\zeta_1^{\pm 1})_\infty(\zeta_2^{\pm 1})_\infty((\zeta_1\zeta_2)^{\pm 1})_\infty \\
&=\frac{\zeta_{1} \zeta_{2}}{(1-\zeta_{1})(1-\zeta_{2})(1-\zeta_{1} \zeta_{2})} \sum_{\substack{x_{1} ,x_{2} \in  \Z \\ x_{1} +1 \equiv x_{2} \pmod{3}}} q^{\frac{x_{1}^2 + x_{1}x_{2} + x_{2}^2}{3} -\frac{1}{3}} \\
&\times \left[ \zeta_{1}^{x_{1}} \zeta_{2}^{x_{1}+x_{2}} + \zeta_{1}^{-x_{1}-x_{2}} \zeta_{2}^{-x_{2}} + \zeta_{1}^{x_{2}} \zeta_{2}^{-x_{1}} - \zeta_{1}^{x_{2}} \zeta_{2}^{x_{1}+x_{2}} - \zeta_{1}^{-x_{1}-x_{2}} \zeta_{2}^{-x_{1}} -\zeta_{1}^{x_{1}} \zeta_{2}^{-x_{2}} \right]
\end{align*}
The power of $q$ is $\frac{\ell^2 -1}{3} \pmod{\ell}$ when $x_{1}^2 + x_{1} x_{2}+ x_{2}^2 \equiv 0 \pmod{\ell}$.  It is well known that when $\ell \equiv 2 \pmod{3}$ we can only have $x^2 + xy + y^2 = (x-y)^2 + 3xy \equiv 0 \pmod{\ell}$ when $x \equiv y \equiv 0 \pmod{\ell}$.  This means we must have $x_{1} \equiv x_{2} \equiv 0 \pmod{\ell}$.  When $\zeta_1$ and $\zeta_2$ are specialized to $\ell$-th roots of unity the term inside the brackets vanishes and so the coefficients $[q^{\ell n + \delta_{k,\ell} + \frac{1}{3}} ]\theta_{A_{2}}$ will vanish as well.

\subsection{$h=14$}
When $k \equiv -14 \pmod{\ell}$ and $\ell \equiv 2 \pmod{3}$ then we have $p_{k} \left( \ell n +7 \cdot \frac{\ell^2 -1}{12} \right) \equiv 0 \pmod{\ell}$.

In this case the root system is 
\begin{equation*}
G_{2} = \{ \pm(1,-1,0), \pm(-1,2,-1), \pm(0,1,-1), \pm(1,0,-1), \pm(2,-1,-1), \pm(1,1,-2) \}
\end{equation*}
which has Euclidean space
\begin{equation*}
E_{G_{2}} = \{x=(x_{1}, x_{2}, x_{3}) \in \R^3: x_{1} + x_{2} + x_{3} =0 \}.
\end{equation*}
We can choose
\begin{align*}
G_{2}^{+} &= \{(1,-1,0), (-1,2,-1), (0,1,-1), (1,0,-1), (2, -1, -1), (1,1,-2) \} \\
&= \{r_{1}, r_{2}, r_{3}=r_{1}+r_{2}, r_{4}=2r_{1} + r_{2}, r_{5}=3r_{1} + r_{2}, r_{6} = 3r_{1} + 2r_{2} \}
\end{align*}
and
\begin{equation*}
F_{G_{2}} = \{r_{1}, r_{2} \}.
\end{equation*}
A calculation shows 
\begin{equation*}
G_{G_{2}} = \{ \pm Id, \pm s_{r_{1}}s_{r_{2}}, \pm s_{r_{1}}s_{r_{3}}, \pm s_{r_{1}}, \pm s_{r_{2}}, \pm s_{r_{3}} \} \cong D_{6}
\end{equation*}
with
\begin{align*}
{\rm{sn}}(\pm Id) &= {\rm{sn}}(\pm s_{r_{1}}s_{r_{2}}) = {\rm{sn}}(\pm s_{r_{1}} s_{r_{3}}) =1 \\
{\rm{sn}}(\pm s_{r_{1}}) &= {\rm{sn}}(\pm s_{r_{2}}) = {\rm{sn}}(\pm s_{r_{3}}) =-1.
\end{align*}
We have $\mathfrak{h}=12$ and $w=(2,1,-3)$ and find
\begin{align*}
W_{G_{2}} &= \left\{ x \in E_{G_{2}} : \frac{(x,r)}{12} \in \Z \ \forall r \in G_{2} \right\} \\
&= \{ (x_{1}, x_{2}, -x_{1}-x_{2}) \in \Z^3 : x_{1} \equiv x_{2} \equiv 0 \pmod{4} \ \text{and} \ x_{1} \equiv x_{2} \pmod{3} \} \\
&= W_{G_{2}, ev}.
\end{align*}
We have $x=(x_{1}, x_{2}, -x_{1}-x_2) = (2x_{1} + x_{2})r_{1} + (x_{1}+x_{2})r_{2}$ so by computing the action of $G_{G_{2}}$ on the simple roots we find
\begin{align*}
\pm Id(x) &= \pm \left[(2x_{1} + x_{2})r_{1} + (x_{1} + x_{2})r_{2} \right] \\
\pm s_{r_{1}} s_{r_{2}}(x) &= \pm \left[(x_{1} - x_{2})r_{1} + x_{1} r_{2} \right] \\
\pm s_{r_{1}} s_{r_{3}}(x) &= \pm \left[ (-x_{1}-2x_{2})r_{1} - x_{2}r_{2} \right] \\
\pm s_{r_{1}}(x) &= \pm \left[ (x_{1} + 2x_{2})r_{1} + (x_{1}+x_{2})r_{2} \right] \\
\pm s_{r_{2}}(x) &= \pm \left[(2x_{1} + x_{2})r_{1} + x_{1}r_{2} \right] \\
\pm s_{r_{3}}(x) &= \pm \left[ (x_{1} - x_{2})r_{1} -x_{2}r_{2} \right].
\end{align*}
We make the changes of variable $z_{1} = \frac{(r_{1},\bz)}{12}$ and $z_{2} = \frac{(r_{2},\bz)}{12}$ to obtain
\begin{align*}
&\theta_{G_{2}}(\bz; \tau) = \eta (\tau)^{-4} \theta(z_{1}; \tau) \theta(z_{2}; \tau) \theta(z_{1} + z_{2}; \tau) \theta(2z_{1} + z_{2}; \tau) \theta(3z_{1}+z_{2}; \tau) \theta(3z_{1} + 2z_{2}; \tau) \\
&=q^{\frac{7}{12}} \frac{(1-\zeta_{1})(1-\zeta_{2})(1-\zeta_{1} \zeta_{2})(1-\zeta_{1}^2 \zeta_{2})(1-\zeta_{1}^3 \zeta_{2})(1-\zeta_{1}^3 \zeta_{2}^2)}{\zeta_{1}^{5} \zeta_{2}^3} \\
&\times (q)_\infty^2(\zeta_1^{\pm 1})_\infty(\zeta_2^{\pm 1})_\infty
\\
&\times(1- (\zeta_{1} \zeta_{2})^{\pm 1} q^n)(1-(\zeta_{1}^2 \zeta_{2})^{\pm 1} q^n)(1-(\zeta_{1}^3 \zeta_{2})^{\pm 1} q^n)(1-(\zeta_{1}^3 \zeta_{2}^2)^{\pm 1} q^n) \\
&= \sum_{\substack{x_{1}, x_{2} \in \Z \\ x_{1}  \equiv x_{2} +1 \equiv 2 \pmod{4} \\ x_{1} \equiv x_{2}+1 \pmod{3}}} q^{\frac{x_{1}^2 + x_{1} x_{2} + x_{2}^2}{12}} \\ 
&\times \left[\zeta_{1}^{2x_{1} + x_{2}} \zeta_{2}^{x_{1} + x_{2}} + \zeta_{1}^{-2x_{1} - x_{2}} \zeta_{2}^{-x_{1} - x_{2}} + \zeta_{1}^{x_{1} -x_{2}} \zeta_{2}^{x_{1}} + \zeta_{1}^{-x_{1} + x_{2}} \zeta_{2}^{-x_{1}} + \zeta_{1}^{-x_{1} - 2x_{2}} \zeta_{2}^{-x_{2}} + \zeta_{1}^{x_{1}+2x_{2}} \zeta_{2}^{x_{2}} \right. \\
&\left. - \zeta_{1}^{x_{1} +2x_{2}} \zeta_{2}^{x_{1} + x_{2}} - \zeta_{1}^{-x_{1} -2x_{2}} \zeta_{2}^{-x_{1}-x_{2}} - \zeta_{1}^{2x_{1} + x_{2}} \zeta_{2}^{x_{1}} - \zeta_{1}^{-2x_{1} -x_{2}} \zeta_{2}^{-x_{1}} -\zeta_{1}^{x_{1} - x_{2}} \zeta_{2}^{-x_{2}} - \zeta_{1}^{-x_{1}+x_{2}} \zeta_{2}^{x_{2}} \right].
\end{align*}
Thus we have
\begin{align*}
& (q)_\infty^2(\zeta_1^{\pm 1})_\infty(\zeta_2^{\pm 1})_\infty
\\
&\times(1- (\zeta_{1} \zeta_{2})^{\pm 1} q^n)(1-(\zeta_{1}^2 \zeta_{2})^{\pm 1} q^n)(1-(\zeta_{1}^3 \zeta_{2})^{\pm 1} q^n)(1-(\zeta_{1}^3 \zeta_{2}^2)^{\pm 1} q^n) \\
&=\frac{\zeta_{1}^{5} \zeta_{2}^3}{(1-\zeta_{1})(1-\zeta_{2})(1-\zeta_{1} \zeta_{2})(1-\zeta_{1}^2 \zeta_{2})(1-\zeta_{1}^3 \zeta_{2})(1-\zeta_{1}^3 \zeta_{2}^2)} \sum_{\substack{x_{1}, x_{2} \in \Z \\ x_{1}  \equiv x_{2} +1 \equiv 2 \pmod{4} \\ x_{1} \equiv x_{2}+1 \pmod{3}}} q^{\frac{x_{1}^2 + x_{1} x_{2} + x_{2}^2}{12} - \frac{7}{12}} \\ 
&\times \left[\zeta_{1}^{2x_{1} + x_{2}} \zeta_{2}^{x_{1} + x_{2}} + \zeta_{1}^{-2x_{1} - x_{2}} \zeta_{2}^{-x_{1} - x_{2}} + \zeta_{1}^{x_{1} -x_{2}} \zeta_{2}^{x_{1}} + \zeta_{1}^{-x_{1} + x_{2}} \zeta_{2}^{-x_{1}} + \zeta_{1}^{-x_{1} - 2x_{2}} \zeta_{2}^{-x_{2}} + \zeta_{1}^{x_{1}+2x_{2}} \zeta_{2}^{x_{2}} \right. \\
&\left. - \zeta_{1}^{x_{1} +2x_{2}} \zeta_{2}^{x_{1} + x_{2}} - \zeta_{1}^{-x_{1} -2x_{2}} \zeta_{2}^{-x_{1}-x_{2}} - \zeta_{1}^{2x_{1} + x_{2}} \zeta_{2}^{x_{1}} - \zeta_{1}^{-2x_{1} -x_{2}} \zeta_{2}^{-x_{1}} -\zeta_{1}^{x_{1} - x_{2}} \zeta_{2}^{-x_{2}} - \zeta_{1}^{-x_{1}+x_{2}} \zeta_{2}^{x_{2}} \right].
\end{align*}
As before, the only way to get the power of $q$ to be $7 \frac{\ell^2 -1}{12} \pmod{\ell}$ with $\ell \equiv 2 \pmod{3}$ is when $x_{1} \equiv x_{2} \equiv 0 \pmod{\ell}$.  When this is the case the term inside the bracket vanishes when $\zeta_{1}$ and $\zeta_{2}$ are set to $\ell$-th roots of unity and so the coefficients $[q^{\ell n +\delta_{k, \ell} + \frac{7}{12} }]\theta_{G_{2}}$ also vanish.

\section{Proof of Theorem \ref{MainThm}}\label{MainThmSec}
The main idea for the proof of Theorem \ref{MainThm} is to multiply the numerator and denominator by a theta block.  If there is a congruence coming from $k + h = \ell t$ for a prime $\ell$, then we will multiply our crank function by $\phi_R/\phi_R$ for the corresponding theta block $\phi_R$ determined by the value $h$ in Table~\ref{TableThetaBlocks}.  When this happens, the denominator will be of the form $1-q^{\ell n} + \Phi_{\ell}(\zeta)f(z;\tau)$ for some function $f$.  Therefore, when we set $\zeta$ equal to an $\ell$-th root of unity, the denominator will only be supported on powers of $q^{\ell}$.  The only terms in the numerator will be those given by the theta block, so the congruences are therefore explained due to Lemma~\ref{vanishingcoeffs}.

We begin by showing that the denominator of our crank generating function multiplied by the appropriate theta block is only supported on powers of $q^{\ell}$ modulo $\Phi_\ell(\zeta)$. 

Our first lemma is for the crank generating function given by \eqref{first def}.

\begin{Lemma}\label{denomfactors1}
	Suppose $\{\ell n+\delta_{k,\ell}\}_{n\in\Z_{\geq0}}$ is a generic congruence progression for $P_k(\tau)$ with $k + h = \ell t$ and $h \in \{4,6,8,10\}$. Then if $\phi_R(\bz;\tau)$ is the theta block associated to $h$ in Table~\ref{TableThetaBlocks}, then there is a choice of $a, b \in \Z$ such that if $\bz = (az, bz)$, then  
	\begin{equation}\label{denom}
		\phi_R(\bz;\tau) (\zeta^{\pm (2 - \delta_{\text{odd}}(k))}q)_\infty \cdot \ldots \cdot (\zeta^{\pm (k-2)}q)_\infty (\zeta^{\pm k}q)_\infty \equiv q^{h/24}f(\zeta)(q)_\infty^{\delta_{\text{odd}}(k)}(q^{\ell})_\infty^t \pmod{\Phi_\ell(\zeta)}
	\end{equation}
	for some $f(\zeta) \in \Q[\zeta, \zeta^{-1}]$. 
\end{Lemma}

\begin{proof}
We begin with the observation that
\begin{equation} \label{pre-p}
(1-q^n)(1-\zeta q^n)(1-\zeta^2 q^n) \cdots (1-\zeta^{\ell - 1} q^n) \equiv 1-q^{\ell n} \pmod{\Phi_\ell(\zeta)}.
\end{equation}
Using that $\Phi_\ell(\zeta)$ is the minimal polynomial for the $\ell^{\text{th}}$ roots of unity, we also have that $q^j \equiv q^m\pmod{\Phi_{\ell}(\zeta)}$ whenever $\zeta^j \equiv\zeta^m \pmod{\ell}$, so it turns out that we can generalize \eqref{pre-p}
\begin{equation}\label{p}
	(1-\zeta^{c_0}q^n)(1-\zeta^{c_1}q^n)\cdots(1-\zeta^{c_{\ell-1}}q^n) \equiv 1-q^{\ell n}\pmod{\Phi_{\ell}(\zeta)}
\end{equation}
for any complete set of residues $c_0, c_1, \ldots, c_{\ell - 1}$ modulo $\ell$.
If $k+h = \ell t$, then we will choose $a$ and $b$ so that when $z_1 = az$ and $z_2 = bz$ in our theta block, multiplying the proposed crank generating function by $\phi_R/\phi_R$ from Table \ref{TableThetaBlocks} (which will result in the denominator consisting of $t$ products of factors of the form \eqref{p}). We will cover each case separately, determining which factors need to be filled in by the theta block to complete these $t$ products. We will restrict ourselves to the case when $k$ is odd since the proof when $k$ is even is similar.

For $h=10$, the theta block we are considering is $\phi_{B_2}(\bz;\tau)$ by Table~\ref{TableThetaBlocks}. By using the product representations of $\theta(z;\tau)$ and $\eta(\tau)$ given by \eqref{theta} and \eqref{eta} respectively, we can represent this as the product
\begin{equation*}
\phi_{B_{2}}(\bz; \tau) = q^{5/12}f(\zeta)(q)_\infty^2(\zeta^{\pm a}q)_\infty (\zeta^{\pm b}q)_\infty (\zeta^{\pm (a+b)}q)_\infty (\zeta^{\pm (a+2b)}q)_\infty
\end{equation*}
for some $f(\zeta) \in \Q[\zeta, \zeta^{-1}]$. As a result, we see that the left-hand side of \eqref{denom} is
\begin{equation}\label{intermediatedenom}
q^{5/12}f(\zeta)(q)_\infty^2(\zeta^{\pm a}q)_\infty (\zeta^{\pm b}q)_\infty (\zeta^{\pm (a+b)}q)_\infty (\zeta^{\pm (a+2b)}q)_\infty (\zeta^{\pm 1}q)_\infty \cdot \ldots \cdot (\zeta^{\pm (k-2)}q)_\infty (\zeta^{\pm k}q)_\infty.
\end{equation}
In order to obtain the right-hand side of \eqref{denom}, we factor out $q^{5/12}f(\zeta)(q)_\infty$ of \eqref{intermediatedenom} and rewrite the remaining terms as
\begin{align}\label{prodexpansion}
(q)_\infty&(\zeta^{\pm a}q)_\infty (\zeta^{\pm b}q)_\infty (\zeta^{\pm (a+b)}q)_\infty (\zeta^{\pm (a+2b)}q)_\infty (\zeta^{\pm 1}q)_\infty \cdot \ldots \cdot (\zeta^{\pm (k-2)}q)_\infty (\zeta^{\pm k}q)_\infty.
\end{align}
We now wish to apply \eqref{p} to \eqref{prodexpansion}. To do so, we need
\begin{equation*}
0, \pm a, \pm b, \pm (a+b), \pm(a+2b), \pm 1, \pm 3, \ldots, \pm (k-2), \pm k
\end{equation*}
to form $t$ complete sets of residues modulo $\ell$. As $k+h = \ell t$, a brute force check quickly shows that we need $\pm a, \pm b, \pm (a+b), \pm (a+2b)$ to form the residues $\pm (2,4,6,8) \pmod{\ell}$. By choosing
\begin{equation*}
(a,b,a+b,a+2b) = (4,2,6,8),
\end{equation*}
we have $t$ complete sets of residues. As a result, we may apply \eqref{p} to \eqref{prodexpansion} and see that it is equivalent to $(q^{\ell})_\infty^t \pmod{\Phi_{\ell}(\zeta)}$.

We use the same proof technique in the other cases as well. When $h=8$, the missing residues are
\begin{equation*}
\pm (2,4,6) \pmod{\ell},
\end{equation*}
and we are using the theta block $\theta_{A_{2}}(\bz; \tau)$. To fill in these residues, we may choose
\begin{equation*}
(a,b,a+b) = (2,4,6)
\end{equation*}
to get a set of $t$ complete residues modulo $\ell$. This will allow us to complete the proof in the same way using \eqref{p}. Similarly, when $h=6$ and $h=4$, the missing residues are
\begin{equation*}
\pm (2,4) \pmod{\ell}
\end{equation*}
and $\pm 2$ respectively. Both of these theta blocks allow us specialize each elliptic variable independent of the others, so we can fill in the needed products.   
\end{proof}

Note that when $h=14$, this crank function does not necessarily explain the corresponding congruence because the missing residues are
\begin{equation*}
\pm(2,4,6,8,10,12) \pmod{\ell},
\end{equation*}
and there is no choice of
\begin{equation*}
(a,b,a+b,2a+b,3a+b,3a+2b)
\end{equation*}
coming from $\theta_{G_{2}}(\bz; \tau)$ that always fills these in. This is why we have to consider the case where there is a prime $\ell \equiv 2 \pmod{3}$ that divides $k+14$ separately.

\begin{Lemma}\label{denomfactors2}
	Suppose $\{\ell n+\delta_{k,\ell}\}_{n\in\Z_{\geq0}}$ is a generic congruence progression for $P_k(\tau)$ with $k+h = \ell t$ and $h \in \{4,6,10,14\}$. Then if $\phi_R(\bz;\tau)$ is the theta block associated to $h$ in Table~\ref{TableThetaBlocks}, then there is a choice of $a, b \in \Z$ such that if $\bz = (az, bz)$, then  
	\begin{equation*}
		\phi_R(z;\tau) (\zeta^{\pm (2 - \delta_{\text{odd}}(k))}q)_\infty \cdot \ldots \cdot (\zeta^{\pm (k-2)}q^n)_\infty (\zeta^{\pm (k+2)}q)_\infty \equiv q^{h/24}f(\zeta)(q)_\infty^{\delta_{\text{odd}}(k)}(q^{\ell})_\infty^t \pmod{\Phi_\ell(\zeta)}
	\end{equation*}
	for some $f(\zeta) \in \Q[\zeta, \zeta^{-1}]$. 
\end{Lemma}

\begin{proof}
	This proof uses the same methods as the proof of Lemma~\ref{denomfactors1}, where we are able to use \eqref{p} in order to prove our equivalence. We again restrict ourselves to the case when $k$ is odd and consider case-by-case how to obtain a set of $t$ complete residues. When $h=14$, the missing residues are now
	\begin{equation*}
		\pm (2,4,6,8,10,14) \pmod{\ell},
	\end{equation*}
	and the theta block we multiply by is $\phi_{G_{2}}(\bz; \tau)$. Hence, we can choose
	\begin{equation*}
		(a,b,a+b, 2a+b, 3a+b, 3a+2b) = (2,4,6,8,10,14)
	\end{equation*}
	to fill in the missing residues. When $h=10$, the missing terms are
	\begin{equation*}
		\pm (2,4,6,10) \pmod{\ell},
	\end{equation*}
	and the theta block we multiply by is $\phi_{B_{2}}(\bz; \tau)$. Hence, we can choose
	\begin{equation*}
		(a,b,a+b, a+2b) = (2,4,6,10)
	\end{equation*}
	to fill in the missing residues. Just as in the proof of Lemma~\ref{denomfactors1}, the cases $h =4,6$ are seen to be trivial. 
\end{proof}

Note that when $h=8$, the missing residues are
\begin{equation*}
	\pm(2,4,8) \pmod{\ell}.
\end{equation*}
Here, we have the theta block $\phi_{A_{2}}(\bz; \tau)$, so we would need to fill $\pm (2,4,8)$ using
\begin{equation*}
	(a,b,a+b),
\end{equation*}
which cannot be done in general. This is why we need the separate cases of Lemma~\ref{denomfactors1} and Lemma~\ref{denomfactors2}.

We also note that both definitions of a crank function explain the congruences coming from $k \equiv -3 \pmod{\ell}$ and $k \equiv -1 \pmod{\ell}$ using the same factoring argument for the denominator after multiplying by the appropriate function. The first follows from the Jacobi triple product while the second follows from Euler's pentagonal number theorem.  Definition \eqref{first def} also explains the trivial congruences when $k \equiv 0 \pmod{\ell}$ as the denominator in that case already factors by using the Freshman's Dream.

\begin{proof}[Proof of Theorem~\ref{MainThm}]
Combining \eqref{crankprod}, \eqref{first def}, \eqref{second def}, Lemma~\ref{denomfactors1}, and Lemma~\ref{denomfactors2}, we have the equivalence
\begin{equation*}
\mathcal{C}_k(z;\tau) = \frac{\phi_R(\bz;\tau)\mathcal{C}_k(z;\tau)}{\phi_R(\bz;\tau)} \equiv \frac{q^{-\frac{h}{24}}\phi_R(\bz;\tau)}{f(\zeta)(q^\ell)^t_\infty} \pmod{\Phi_{\ell}(\zeta)}.
\end{equation*}
Since $f(\zeta)(q^{\ell})_\infty$ is only supported on powers of $q^{\ell}$, we see that $[q^{\ell n + \delta_{k,\ell}}]\mathcal{C}_k(z;\tau)$ is a sum of terms of the form $[q^{\ell n + \delta_{k,\ell}}]\phi_R(z;\tau)$, which are $0$ by Lemma~\ref{vanishingcoeffs}. Hence, we conclude that $[q^{\ell n + \delta_{k,\ell}}]\mathcal{C}_k(z;\tau) \equiv 0 \pmod{\Phi_{\ell}(\zeta)}$.
\end{proof}

\section{Examples and discussion}\label{FinalSection}

While we have proved that the crank generating functions \eqref{first def} and \eqref{second def} explain most of the congruences coming from Theorem~\eqref{c} simultaneously, there may be other choices of $a_1, a_2, \dots a_{\frac{k+\delta_{\text{odd}}}{2}}$ in \eqref{general Ck} that explain these congruences as well. In this section we give examples showing how one might find a crank function for any given color using the theory of theta blocks.  In particular, both $k=3$ and $k=9$ have singular congruences, but there is a crank function that explains all the Ramanujan-like congruences for these colors simultaneously for each $\ell$.  For $k=5$ we show that the choice given in \eqref{first def} works to explain all of the congruences, but there are actually three other choices that work as well.  We have also highlighted $k=33$ as the first color where the authors have not been able to identify a crank which explains all of the congruences.

\subsection{$k=3$} In this first case, we give a more detailed description of how to find other crank generating functions that explain the congruences of Theorem~\ref{c}. Recall that we are only studying crank generating functions of the form \eqref{general Ck}. There are of course cranks that could exist whose generating functions are not of this form, but our reason for studying these is that they have nice representations as products and quotients of Pochhammer symbols. Possibly after cancellation, we will be able to group terms in such a way that we can use \eqref{p} to see that the denominator is only supported on powers of $q^{\ell}$ modulo $\Phi_\ell (\zeta)$. When $k=3$ for example, this generating function will be of the form 
\begin{equation*}
	\mathcal{C}(0;\tau)\mathcal{C}(a_1 z;\tau)\mathcal{C}(a_2 z;\tau) = \frac{(q)_\infty}{(\zeta^{\pm a_1}q)_\infty (\zeta^{\pm a_2 }q)_\infty}. 
\end{equation*}
The partition congruences in this case come from 
\begin{align*}
	3 &\equiv -8 \pmod{11},\\
	3 &\equiv - 14 \pmod{17}.
\end{align*}
In other words, we have $p_3(11n + 7) \equiv 0 \pmod{11}$ and $p_3(17n + 15) \equiv 0 \pmod{17}$. In order to use Lemma~\ref{vanishingcoeffs}, we will multiply by the numerator and denominator by a theta block $\phi_R$ determined by Table~\ref{TableThetaBlocks}. Because $3 \equiv -8 \pmod{11}$, this means that the theta block for $\ell = 11$ is
\begin{equation*}
	\phi_{A_2}(\bz;\tau) = (q)_\infty (\zeta^{\pm a} q)_\infty (\zeta^{\pm b}q)_\infty (\zeta^{\pm (a + b)}q)_\infty.
\end{equation*}
What this tells us is that we need
\begin{equation*}
	0, \pm a, \pm b \pm (a + b), \pm a_1, \pm a_2
\end{equation*}
to be a complete set of residues for some choices of $a, b, a_1, a_2 \in \Z$ in order for the denominator to factor as we would like. Similarly, the theta block $\phi_R$ coming from $\ell = 17$ is
\begin{equation*}
	\phi_{G_2}(\bz;\tau) = (q)_\infty (\zeta^{\pm a'} q)_\infty (\zeta^{\pm b'}q)_\infty (\zeta^{\pm (a' + b')}q)_\infty (\zeta^{\pm (2a' + b')}q)_\infty (\zeta^{\pm (3a' + b')}q)_\infty (\zeta^{\pm (3a'+ 2b')}q)_\infty.
\end{equation*}
Hence, we need
\begin{equation*}
	0, \pm a', \pm b', \pm (a' + b'), \pm (2a' + b'), \pm(3a' + b'), \pm(3a' + 2b'), \pm a_1, \pm a_2
\end{equation*}
to form a complete set of residues mod $17$ for some choice of $a', b', a_1, a_2 \in \Z$. Notice that the choices of $a,b$ and $a',b'$ may be different since the theta blocks are different, but the choice of $a_1,a_2$ must be the same since we want the crank generating function we are trying to find to be the same for both of these congruences. To determine if such a function exists, it actually only takes a finite check since $a,b$ are only unique mod $11$, $a',b'$ are only unique mod $17$, and $a_1,a_2$ are only unique mod $11\cdot 17$. We decided to narrow down the search space for many of our checks by letting $a_1 = 3 > a_2$. This restriction can of course be relaxed, but we needed to create fewer choices for higher values of $k$ later on using a similar restriction. With this caveat, we found the only possible option for $k = 3$ that creates a crank is when $a_2 = 2$. By letting $a = 1$, $b = 4$, $a' = 1$, $b' = 5$, we see that
\begin{equation*}
	\mathcal{C}(3z, 2z;\tau) = \frac{(q)_\infty}{(\zeta^{\pm 2}q)_\infty (\zeta^{\pm 3}q)_\infty}
\end{equation*}
is a crank generating function for $k = 3$ which explains the two congruences simultaneously.

\subsection{$k=5$}
We have 
\begin{align*}
5 &\equiv 0 \pmod{5}, \\
5 & \equiv -6 \pmod{11},
\end{align*} 
so we want to look at
\begin{equation*}
\theta_{A_{1} \oplus A_{1}}(\bz; \tau) = \theta(z_{1}; \tau) \theta(z_{2}; \tau),
\end{equation*}
We have a lot of freedom here as we can choose $a$ and $b$ to be any integers.  We choose $a_{1} =5$ and want to choose $a_{2}$ and $a_{3}$ such that $\pm a_{2}$ and $\pm a_{3}$ are all in different residue classes modulo $\pmod{11}$.  We also have $5 \equiv 0 \pmod{5}$ so we want to choose $\pm a_{2}, \pm a_{3} \not\equiv 0 \pmod{5}$ to be in different residue classes modulo $5$.  Therefore the choices for the $5$-colored crank are
\begin{equation*}
\mathcal C(5z,4z,3z; \tau), \quad \mathcal C(5z,4z,2z; \tau), \quad \mathcal C(5z,3z,z; \tau), \quad \mathcal C(5z,2z,z; \tau).
\end{equation*}

\subsection{$k=9$} \label{k=9}
We have
\begin{align*}
9 & \equiv -1 \pmod{5}, \\
9 & \equiv -8 \pmod{17}, \\
9 &\equiv -10 \pmod{19}, \\
9 & \equiv -14 \pmod{23},
\end{align*}
so we look at the theta blocks
\begin{align*}
\theta_{A_{2}}(\bz; \tau) &= \eta(\tau)^{-1} \theta(z_{1}; \tau) \theta(z_{2}; \tau) \theta(z_1 + z_2; \tau), \\
\theta_{B_{2}}(\bz; \tau) &= \eta(\tau)^{-2}\theta(z_{1}; \tau) \theta(z_{2}; \tau) \theta(z_{1} + z_{2}; \tau) \theta(z_{1} + 2z_{2}; \tau), \\
\theta_{G_{2}}(\bz; \tau) &= \eta(\tau)^{-4} \theta(z_{1}; \tau) \theta(z_{2}; \tau) \theta(z_{1} + z_{2}; \tau) \theta(2z_{1} + z_{2}; \tau) \theta(3z_{1} + z_{2}; \tau) \theta(3z_{1} + 2z_{2}; \tau).
\end{align*}
When multiplying by $\theta_{G_2}$, we need to choose the $a_i$ so that $\pm(a,b,a+b, 2a+b, 3a+b, 3a+2b)$ all live in different residue classes modulo $23$ and all avoid $\pm 9 \pmod{23}$.  The possible choices are
\begin{align*}
(a,b,a+b, 2a+b, 3a+b, 3a+2b) &= (3,1,4,7,10,11), \\
(a,b,a+b, 2a+b, 3a+b, 3a+2b) &= (3,2,5,8,11,13) \equiv (3,2,5,8,11,-10) \pmod{23}, \\
(a,b,a+b, 2a+b, 3a+b, 3a+2b) &=(8,5,13,21,29,34) \equiv (8,5,-10,-2,6,11) \pmod{23},
\end{align*}
which leaves the choices
\begin{align*}
(a_{1}, a_{2}, a_{3}, a_{4}, a_{5}) &= (9,8,6,5,2), \\
(a_{1}, a_{2}, a_{3}, a_{4}, a_{5}) &=(9,7,6,4,1), \\
(a_{1}, a_{2}, a_{3}, a_{4}, a_{5}) &=(9,7,4,3,1)
\end{align*}
for the $9$-colored crank if we just look modulo $23$.  Modulo $19$ we multiply by $\theta_{B_{2}}$ so we need $\pm(a,b,a+b,a+2b)$ to all be in different residue classes modulo $19$ and avoid the possible choices above.  All three possible choices still work because we can choose
\begin{align*}
(a,b,a+b,a+2b) &= (1,3,4,7), \\
(a,b,a+b,a+2b) &=(2,3,5,8), \\
(a,b,a+b,a+2b) &=(2,6,8,14) \equiv (2,6,8,-5) \pmod{19}.
\end{align*}
Finally, we look modulo $17$ at $\theta_{A_{2}}$.  Notice that $(9,8,6,5,2) \equiv (-8,8,6,5,2) \pmod{17}$, so it can be eliminated as a choice because $\pm 9$ and $\pm 8$ will fill out the same moduli classes modulo $17$ (and thus will prevent us from having a complete set of residues).  The other choices become
\begin{align*}
(9,7,6,4,1) &\equiv (-8,7,6,4,1) \pmod{17}, \\
(9,7,4,3,1) &\equiv (-8,7,4,3,1) \pmod{17},
\end{align*}
so we need $\pm (a,b,a+b) \equiv \pm(2,5,6)$ or $\pm (2,3,5) \pmod{17}$.  The first one is not possible while the second one is so we must have the only possible crank explaining all of these congruences is
\begin{equation*}
\mathcal C(9z, 7z, 6z, 4z, z; \tau).
\end{equation*}  
We note that this crank function also explains the congruences $p_{9}(5n+\delta) \equiv 0 \pmod{5}$ for $\delta =3,4$ in the following way.  We multiply the numerator and denominator of $\mathcal C(9z,7z,6z,4z,z; \tau)$ by $(q)_\infty(\zeta^{\pm1}q)_\infty^2$ to obtain
\begin{equation*}
\frac{(q)_\infty^2(\zeta^{\pm1}q)_\infty^2}{(q)_\infty(\zeta^{\pm1}q)_\infty^3(\zeta^{\pm4}q)_\infty(\zeta^{\pm6}q)_\infty(\zeta^{\pm7}q)_\infty(\zeta^{\pm9}q)_\infty}.
\end{equation*}
When we set $\zeta=\zeta_{5}$ each term of the denominator becomes 
\begin{equation*}
(1-\zeta_{5}^{\pm 1} q^n)^5 (1-q^n)(1-\zeta_{5}^{\pm 1}q^n)(1-\zeta_{5}^{\pm 2} q^n) \equiv (1-q^{5n})^2 \pmod{5}.
\end{equation*}
Meanwhile, the numerator was discussed in Section \ref{wt1}.  We must have $\frac{\left(\alpha + \frac{1}{2} \right)^2 + \left(\beta + \frac{1}{2} \right)^2}{2} - \frac{1}{4} \equiv 3, 4 \pmod{5}$ in order to get the coefficients $\left[q^{5n+3} \right]$ and $\left[q^{5n+4} \right]$ respectively.  This is equivalent to $(\alpha -2)^2 + (\beta -2)^2 \equiv -1, 1 \pmod{5}$.  In either case one of $\alpha$ or $\beta$ must be $2 \pmod{5}$ which causes the coefficient to vanish.  

\subsection{$k=33$}
We have
\begin{align*}
33 & \equiv 0 \pmod{11}, \\
33 & \equiv -1 \pmod{17}, \\
33 & \equiv -8 \pmod{41}, \\
33 & \equiv -10 \pmod{43},\\
33 & \equiv -14 \pmod{47}, 
\end{align*}
so we want to look at $\theta_{A_{2}}(\bz; \tau), \theta_{B_{2}}(\bz; \tau)$, and $\theta_{G_{2}}(\bz; \tau)$.  When multiplying by $\theta_{G_{2}}$, we need to choose $(a, b, a+b, 2a+b, 3a+b, 3a+2b)$ that all live in different residue classes modulo $47$.  Some examples of choices are
\begin{align*}
 (a, b, a+b, 2a+b, 3a+b, 3a+2b) & \equiv (1,7,8,9,10,17) \pmod{47}, \\
 (a, b, a+b, 2a+b, 3a+b, 3a+2b) & \equiv (2,6,8,10,12,18) \pmod{47}, \\
 (a, b, a+b, 2a+b, 3a+b, 3a+2b) &\equiv (8,2,10,18,-21,-19) \pmod{47}, \\
 (a, b, a+b, 2a+b, 3a+b, 3a+2b) & \equiv (10,8,18,-19,-9,-1) \pmod{47}.
 \end{align*}
Recall that we are trying to choose $a_1, \cdots, a_{17}$ to fill out the residue classes modulo all of the above primes. For $\ell = 47$, this means that  
\begin{align}
	\pm a, \pm b, \pm a+b, &\pm 2a+b, \pm 3a+b, \pm 3a+2b, \pm a_1, \cdots, \pm a_{17} \nonumber \\
	&\equiv \pm 1, \pm 7, \pm 8, \pm 9, \pm 10, \pm 17, \pm a_1, \cdots, \pm a_{17} \pmod{47} \label{Eqn: residue classes}
\end{align}
needs to be a complete set of residues modulo $47$. Working modulo $43\cdot 47$, one may conclude that some $a_i$ needs to be equal to $33\pmod{43\cdot 47}$ in order to fill out the residue classes modulo $43$ (since $10\pmod{43}$ will already by chosen in every possible case). Working modulo $47$ again, this will then mean that another $a_j$ must be $14\pmod{47}$. By continued reasoning in this manner, we end up needing some $a_k \equiv 17\pmod{47}$, which means \eqref{Eqn: residue classes} will not be a complete set of residues modulo $47$. The idea why the remaining cases do not work is similar. For the second choice, having $18$ appear means that we cannot have an $a_{i}$ equal to $18$ or $29$.  However, not having $29$ implies we need to choose $12$ after looking modulo $41$, but this is not available so this choice cannot work.
For the third choice, the $10$ implies one $a_{i}$ must be $33$ which then tells us that $14$ cannot be chosen.  However, $18$ appearing tells us we can't choose $29$ either, which when looking modulo $43$ tells us we must choose $14$ as an $a_{i}$ and so this choice cannot work either.  The exact same argument disqualifies the last choice as well. $k=33$ is the first color where the authors have not been able to identify a crank function that explains  the singular congruences.

\subsection{Colored partitions with $j$ colors into even parts}\label{examples}
As an example of the usefulness of the theory of theta blocks we will give a short example of how to construct crank functions which explain congruences for other partition functions.  For simplicity we will only focus on congruences modulo $5$ and $7$, however the techniques used for the $k$-colored partition function for other moduli should extend naturally here as well.  Define $p_{k,j}(n)$ to be the number of $k+j$ colored partitions of $n$ where $j$ of the colors are partitions into even parts.  The generating function is given by
\begin{equation} \label{2j}
\sum_{n \geq 0} p_{k,j}(n) q^n = \prod_{n \geq 1} \frac{1}{(1-q^n)^k (1-q^{2n})^j}.
\end{equation}
Using standard techniques in the theory of modular forms one can confirm that Table \ref{2ColorCongruences5} and Table \ref{2ColorCongruences7} give the values of $\delta$ such that $p_{k,j}(5n + \delta) \equiv 0 \pmod{5}$ and $p_{k,j}(7n + \delta) \equiv 0 \pmod{7}$ respectively. \\

\begin{center}
\begin{table}[h]
 \begin{tabular}{||c| c| c| c | c | c||} 
 \hline
  $k \diagdown j \pmod{5}$ & 0 & 1 & 2 & 3 & 4 \\ [0.5ex] 
 \hline\hline
 0 & None & 3 & 1, 3, 4 & None & 1, 3 \\ 
 \hline
 1 & 4 & None & None & 2, 4 & None \\
 \hline
 2 & 2, 3, 4 & None & 4 & None & None \\
 \hline
 3 & None & 2, 3 & None & None & None \\
 \hline
 4 & 3, 4 & None & None & None & None \\ [1ex] 
 \hline
\end{tabular}
 \caption{Congruences for $p_{k,j}(5n + \delta) \pmod{5}$.}
\label{2ColorCongruences5}
\end{table}
\end{center}

\begin{center}
\begin{table}[h]
 \begin{tabular}{||c| c| c| c | c | c| c| c||} 
 \hline
  $k \diagdown j \pmod{7}$ & 0 & 1 & 2 & 3 & 4 & 5 & 6 \\ [0.5ex] 
 \hline\hline
 0 & None & 3 & None & None & 1, 3, 4, 5 & None & 1, 5, 6 \\ 
 \hline
 1 & 5 & None & 4 & None & None & 2, 4, 5 & None \\
 \hline
 2 & None & 6 & 2, 3, 4, 6 & None & None & None & None \\
 \hline
 3 & None & None & None & None & None & 2 & None \\
 \hline
 4 & 2, 4, 5, 6 & None & None & None & 4 & None & None \\ 
 \hline
 5 & None & 3, 5, 6 & None & 6 & None & 5 & None \\
 \hline
 6 & 3, 4, 6 & None& None& None& None& None & None \\ [1ex] 
 \hline
\end{tabular}
 \caption{Congruences for $p_{k,j}(7n + \delta) \pmod{7}$.}
\label{2ColorCongruences7}
\end{table}
\end{center}

\vspace{-1cm}

One can use combinations of $\theta$, $\theta^{*}$ and the theta blocks in Table \ref{TableThetaBlocks} shifted by half-periods (i.e. $z \mapsto z + \frac{1}{2}, z+\frac{\tau}{2}$, or $z + \frac{1}{2} + \frac{\tau}{2}$) to construct crank functions for the colored partition congruences listed above.  The last author is investigating these shifted theta blocks further in a follow-up work.  Define
\begin{align*}
\mathcal{C}_{k,j}&(z_{1}, z_{2}, \dots, z_{\frac{k + \delta_{\text{odd}}(k)}{2}}; u_{1}, u_2, \dots, u_{\frac{j + \delta_{\text{odd}}(j)}{2}}; \tau) \\
&:= \mathcal{C}(0; \tau)^{\frac{k-\delta_{\text{odd}}(k)}{2}} \mathcal{C}(0; 2 \tau)^{\frac{j - \delta_{\text{odd}}(j)}{2}} \prod_{i=1}^{\frac{k + \delta_{\text{odd}}(k)}{2}} \prod_{j=1}^{\frac{j + \delta_{\text{odd}}(j)}{2}} \mathcal{C}(z_i ; \tau) \mathcal{C}(u_j ; 2 \tau)
\end{align*}
in analogy to \eqref{general Ck}.  

\begin{center}
\begin{table}[h]
 \begin{tabular}{||c|c||} 
 \hline
  $(k, j) \pmod{\ell}$ & Theta block \\ [0.5ex] 
 \hline\hline
 $(1,3) \pmod{5}$ &  $\theta \left( z + \frac{1}{2}; \tau \right)$ \\ 
 \hline
 $(2,2) \pmod{5}$ & $\theta(z_1; \tau) \theta(z_2; 2 \tau)$  \\
 \hline
 $(3,1) \pmod{5}$ & $\theta \left( z + \frac{\tau}{2}; \tau \right)$ \\
 \hline
 $(1,2) \pmod{7}$ & $\theta^{*} \left( z_1 + \frac{1}{2}; \tau \right) \theta^{*} (z_2; \tau)$ \\
 \hline
 $(1,5) \pmod{7}$ & $\theta \left(z + \frac{1}{2} \right)$ \\ 
 \hline
  $(2,1) \pmod{7}$ & $\theta^{*} \left(z_1 + \frac{1}{2}; \tau \right) \theta^{*}(z_2; 2 \tau)$ \\
 \hline
  $(2,2) \pmod{7}$ & $\theta^{*} \left(z + \frac{1}{2} \right)$ \\
 \hline
  $(3,5) \pmod{7}$ & $\theta^{*} \left( z_1 + \frac{\tau}{2}; \tau \right) \theta^{*} \left( z_2 + \frac{1}{2} + \frac{\tau}{2}; \tau \right)$ \\
 \hline
  $(4,4) \pmod{7}$ & $\theta(z_1; \tau) \theta(z_2; 2 \tau)$ \\
 \hline
  $(5,1) \pmod{7}$ & $\theta(z + \tau; 2 \tau)$ \\
 \hline
  $(5,3) \pmod{7}$ & $\theta^{*} \left(z_1 + \frac{1}{2}; 2 \tau \right) \theta^{*} \left(z_2 + \frac{1}{2} + \tau; 2 \tau \right)$ \\
 \hline
  $(5,5) \pmod{7}$ & $\theta(z_1; \tau) \theta \left( z_2 + \frac{1}{2}; \tau \right)$ \\
 [1ex] 
 \hline
\end{tabular}
 \caption{Shifted theta blocks for $p_{k,j}(n)$.}
\label{ShiftedThetaBlocks}
\end{table}
\end{center}

Table \ref{ShiftedThetaBlocks} gives the shifted theta blocks needed to construct crank functions for $kj \not\equiv 0 \pmod{\ell}$.  These details are left to an interested reader.  The crank function
\begin{align*}
	&\mathcal{C}_{k,j}(z_{1}, z_{2}, \dots, z_{\frac{k + \delta_{\text{odd}}(k)}{2}}; u_{1}, u_2, \dots, u_{\frac{j + \delta_{\text{odd}}(j)}{2}}; \tau) \\
	&=\mathcal{C}_{k,j}(kz, (k-2)z, \dots, (2- \delta_{\text{odd}}(k))z; (2j+k)z, (2j+k-2)z, \dots, (j+k+2-\delta_{\text{odd}}(j))z; \tau)
\end{align*}
simultaneously explains all of the congruences modulo $5$ and $7$ unless $j \equiv -2 \pmod{7}$ or $(k,j) \equiv (5,3) \pmod{7}$.  Crank functions can be constructed to explain each individual congruence, but the authors have not identified a family that simultaneously explains all congruences.  Finding a pattern which gives a natural family has been left as an open problem. 

We note that the choice of theta block in Table \ref{ShiftedThetaBlocks} is not necessarily unique.  For example, when $(k,j) \equiv (5,3) \pmod{7}$ one could also use the theta block $\theta_{B_2} \left( z_1, z_2 +\frac{1}{2}; \tau \right)$ to construct a crank function.

\end{document}